\documentclass[11pt,reqno]{amsart}
\usepackage{amssymb}
\usepackage[curve,matrix,arrow]{xy}
\newtheorem{thm}{Theorem}[section]
\newtheorem{lemma}[thm]{Lemma}
\newtheorem{cor}[thm]{Corollary}
\newtheorem{prop}[thm]{Proposition}

\newtheorem{defi}[thm]{Definition}

\theoremstyle{definition}

\def\CI{{\mathcal{I}}}
\def\CJ{{\mathcal{J}}}

\def\CM{{\mathcal{M}}}

\def\CW{{\mathcal{W}}}
\def\CU{{\mathcal{U}}}

\def\CV{{\mathcal{V}}}

\def\Dim{\operatorname{Dim}\nolimits}
\def\dim{\operatorname{dim}\nolimits}

\def\Rad{\operatorname{Rad}\nolimits}
\def\Soc{\operatorname{Soc}\nolimits}

\def\Hom{\operatorname{Hom}\nolimits}

\def\Dist{\operatorname{Dist}\nolimits}

\title{Endotrivial modules for finite group schemes II}
\author[Jon F. Carlson]{Jon F. Carlson}
\thanks{Research of the first author partially supported by NSF grant
DMS-1001102}
\address{Department of Mathematics, University of Georgia,
Athens, Georgia 30602, USA}
\email{jfc@math.uga.edu}
\author[Daniel K. Nakano]{Daniel K. Nakano}
\thanks{Research of the second author partially supported by NSF grant
DMS-1002135}
\address{Department of Mathematics, University of Georgia,
Athens, Georgia 30602, USA}
\email{nakano@math.uga.edu}

\date\today

\begin{document}

\begin{abstract} It is well known that 
if $G$ is a finite group then the group 
of endotrivial modules is finitely generated.  In this paper we prove 
that for an arbitrary finite group scheme $G$, and for any fixed 
integer $n > 0$, there are only finitely many isomorphism classes of endotrivial modules of dimension $n$. This provides 
evidence to support the speculation that the group of 
endotrivial modules for a finite group scheme 
is always finitely generated. The result also has some applications 
to questions about lifting and twisting the structure of endotrivial modules in the case that $G$ is an infinitesimal group scheme
associated to an algebraic group. 
\end{abstract}

\maketitle

\section{Introduction}

Let $S$ be a finite group scheme over a field $k$. 
The endotrivial modules for $S$ 
form an important class of modules 
which, among other things, determine self equivalences 
of the stable category of $S$-modules, modulo projective
$S$-modules. In the case that $G$ is the scheme of a 
finite $p$-group, there is a complete classification 
of endotrivial modules \cite{CT}. This classification 
has been extended to the group algebras of many
other families of finite groups (cf. \cite{CHM, CMN1,CMN2}).

An endotrivial module is an $S$-module $M$ with the property that 
$\Hom_k(M,M) \cong M\otimes M^{*}\cong k \oplus P$ 
(as $S$-modules) where $P$ is a 
projective $S$-module. Two endotrivial modules $M$ and $N$ are 
equivalent if there exist projective modules $P$ and $Q$ 
such that $M \oplus P \cong N \oplus Q$. The equivalence classes of 
endotrivial $S$-modules forms an abelian 
group $T(S)$ under tensor product. It was shown by Puig \cite{Pu}
that this group is finitely generated 
in the case that $S$ is a finite group. 
This fact also follows from the classification. For arbitrary 
finite group schemes, it is an open question 
as to whether $T(S)$ is finitely generated. 

One of the main ingredients in proving that 
$T(S)$ is finitely generated is a demonstration that 
for any fixed non-negative integer 
$n\geq 0$, there are only finitely many endotrivial modules 
of dimension equal to $n$. When $S$ is a 
unipotent group scheme this was proved by the authors 
in \cite[Theorem 3.5]{CN}. In Section 2, we 
extend these earlier results by showing that 
for arbitrary finite group schemes
there are only finitely many endotrivial 
modules for a given dimension. This 
result will be referred to as 
the ``Finiteness Theorem". 

The Finiteness Theorem has some very strong 
connections with the notion of lifting endotrivial modules 
to action of $H$ where $H$ is a group 
scheme containing $S$ as a normal subgroup scheme. In the 
case when $H$ is connected, the Finiteness 
Theorem implies that every endotrivial $S$-module is 
$H$-stable (i.e., the twists of an endotrivial 
module $M^{h}$ are all isomorphic to $M$ 
(as $S$-module) for all $h\in H$). 
Different notions of lifting such as 
``tensor stability'' and ``numerical stability'' 
have been investigated recently by Parshall 
and Scott \cite{PS}. In Section 3, we outline 
these various definitions, and we 
introduce the new concept of 
lifting called ``stably lifting'' which entails 
lifting $S$-modules to $H$-structures in the stable 
module category for $S$. 
We connect our new notion of stable lifting with 
the ideas presented by Parshall and Scott. 

Let $G$ be semisimple algebraic group scheme, 
$B$ a Borel subgroup with unipotent radical 
$U$ defined and split over ${\mathbb F}_{p}$, and 
let $k$ be an algebraically closed field 
of characterstic $p$.  Let $G_{r}$, $B_{r}$, $U_r$ 
denote their $r^{th}$ infinitesimal 
Frobenius kernels. The existence of the Steinberg 
module shows that every projective $G_{r}$-module 
(resp. $B_{r}$, $U_{r}$) lifts to
$G$ (resp. $B$, $U$). Furthermore, if $M$ stably 
lifts then one can use this result to show that 
all syzygies $\Omega^{n}(M)$ stably lift. 
In \cite[Theorem 5.7, Theorem 6.1]{CN}, $T(B_{1})$ 
and $T(U_{1})$ was completely determined for all 
primes. By using this classification, we prove that 
all $B_{1}$ (resp. $U_{1}$) endotrivial modules 
stably lift to $B$ (resp. $U$). We suspect this 
will also hold for $G_{r}$. 
Finally, we exhibit an endotrivial module for 
$B_{1}$, namely $\Omega^{2}(k)$ when the root system 
is of type $A_{2}$ and $p=2$, which does not 
admit a $B$-structure.   

The second author would like to acknowledge 
Shun-Jen Cheng, Weiqiang Wang and their 
conference organizing team for providing financial 
support and for efforts in hosting a first rate workshop 
on representation theory in Taipei during December 2010. 

\section{The finiteness theorem}

We begin by introducing the basic definitions 
which will be used throughout this paper. 
Let $S$ be a finite group scheme. We will 
consider the category $\text{Mod}(S)$ of 
rational $S$-modules and the stable module 
category $\text{StMod}(S)$. Since $S$ 
is a finite group scheme the notion of 
projectivity is equivalent to injectivity. For $M$ 
and $N$ objects in $\text{Mod}(S)$,  
we say that $[M]=[N]$ in $\text{StMod}(S)$ if 
and only if $M\oplus P \cong N\oplus Q$ for some
projective $S$-modules, $P$ and $Q$. 

Suppose that $S$ is any finite group scheme 
defined over $k$. An $S$-module
is an endotrivial module provided that, as $S$-modules,
\[
\Hom_k(M,M) \cong k \oplus P
\]
for some projective $S$-module $P$. 
Equivalently, an $S$-module $M$ is endotrivial 
module if $[M^{*}\otimes M]=[k]$ in $\text{StMod}(S)$.  
Note for any $S$-module $M$ there exists a 
canonical isomorphism $\Hom_k(M,M) \cong M^* \otimes M$. 
We can now define the group $T(S)$ of endotrivial 
$S$-modules as follows. 
The objects in $T(S)$ are the equivalence classes 
$[M]$ in $\text{StMod}(S)$
of endotrivial $S$-modules. The group operation is given
by $[M] + [N] = [M \otimes N]$. The identity 
element is the class $[k]$
and the inverse of $[M]$ is the class $[M^*]$. 
The group $T(S)$ is abelian because 
the associated coproduct used to construct 
the action of $S$ on the tensor product of modules 
is cocommutative.  

In this section, we show that the number of 
endotrivial modules for a finite group scheme 
$H$ having any particular
dimension is finite. The proof follows somewhat the same
lines if that in \cite{CN} and also \cite{Pu}, and is based
on an idea of Dade \cite{D}.

Suppose that $k$ is an algebraically closed field. 
In this section it will be 
more convenient to work with finite-dimensional 
cocommutative Hopf algebras. 
For a finite group scheme $H$, let $A=k[H]^{*}$, 
the group algebra of $H$. 
There is an equivalence of categories 
between $H$-modules and $A$-module. Furthermore, 
the finite dimensional $k$-algebra $A$ 
is a cocommutative Hopf algebra. As a consequence, 
projective $A$-modules are also injective. 
Let $P_1, \dots, P_r$ be a complete set of representatives
of the isomorphism classes of projective indecomposable
$A$-modules. Each $P_i$ has a simple top $S_i = P_i/\Rad(P_i)$ 
and a simple Socle $T_i = \Soc(P_i)$. 
The collection $\{S_1, \dots, S_r\}$
is a complete set of isomorphism classes of simple modules,
as is the set $\{T_1, \dots, T_r\}$. We need not assume here
that $T_i \cong S_i$, though this is often the case.

For each $i = 1, \dots, r$, we assume that $P_i$ is a left
ideal in $A$. That is, we assume that $P _i = Ae_i$ for $e_i$ a
primitive idempotent in $A$. For each $i$ choose a nonzero element 
$u_i \in T_i \subseteq P_i$. Then $u_i$ has the property that $Au_i 
= T_i$. Note that since $u_i \in P_i$, we have that $u_i = u_ie_i$ and
that $u_iP_i \neq \{0\}$.

\begin{lemma}
Let $i$ be an integer between 1 and $r$.
Suppose that $M$ is an $A$-module and that $u_iM \neq \{0\}$. 
Then $M$ has a direct summand isomorphic to $P_i$. Moreover,
if $t_i = \Dim(u_iP_i)$ is the rank of the 
operator of left multiplication
by $u_i$ on $P_i$ and $a_i = \Dim(u_iM)/t_i$, then 
\[ 
M \ \cong \ P_i^{a_i} \oplus M^\prime
\]
where $M^\prime$ has no direct summands isomorphic to $P_i$.
\end{lemma}

\begin{proof}
Let $m \in M$ be an element such that $u_im \neq 0$. Then, 
$u_im = u_ie_im$. Define $\psi: P_i \longrightarrow M$ by 
$\psi(ae_i) = ae_im$ for any $a$ in $A$. This is well defined 
since $P_i = Ae_i$. Now, $\psi(u_i) \neq 0$ and hence 
$\psi(T_i) \neq \{0\}$. Therefore, $\psi$ is injective. Because,
$A$ is a self-injective ring, the image of $\psi$ is a direct 
summand of $M$ and hence $M \cong P_i \oplus N$ for 
some submodule $N$ of $M$. This proves the first statement.

The second statement follows by an easy induction beginning
with the observation that (as vector spaces)
\[
u_iM \ \cong \ u_iP_i \oplus u_iN.
\]
\end{proof}

For a positive integer $n$, let  $\CV_n$ denote 
the variety consisting of all
representations of the algebra $A$ of dimension $n$.
It is defined as follows. Suppose that the collection 
$a_1,  \dots, a_t$ is a chosen set of generators 
of the algebra $A$, so that
every element of $A$ can be written as a polynomial  
in (noncommuting) variables $a_1, \dots, a_t$. We fix this
set of generators for the remainder of the discussion in 
this section. 
Then we have that $A \cong k\langle a_1,
\dots, a_t \rangle/\CJ$ for some ideal $\CJ$. 
A representation of dimension
$n$ of $A$ is a homomorphism 
$\theta: A \longrightarrow M_n(k)$, 
where $M_n(k)$ is the ring of $n \times n$ 
matrices over $k$. The 
representation is completely determined by the assignment to each 
$a_i$ of an $n \times n$ matrix $\theta(a_i) = (a^i_{rs})$. 

For the purposes of defining the variety $\CV_n$, we consider the
polynomial ring $R = k[x_{rs}^i]$ in $tn^2$ (commuting) 
variables with $1 \leq i \leq t$, 
$1 \leq r,s \leq n$, and the assignment 
\[ 
a_i \quad \leftrightarrow \quad (x_{rs}^i) \ \in M_n(R)
\]
for $i = 1, \dots, t$. The ideal $\CJ$ determines an ideal $\CI$ in 
the ring $R$.   That is, any relation $f(a_1, \dots, a_t)$
in $\CJ$, when converted into an expression on the matrices using
the above assignment, defines a collection of relations, 
one for each $r$ and $s$ 
in the elements of $R$. For example, if it 
were the case that $a_1a_2 =0$ in 
$A$, then the polynomial $a_1a_2$ would be an element of $\CJ$, 
and for each $r$ and $s$, the polynomial 
$\sum_{u = 1}^n x_{ru}^1x_{us}^2$ would
be an element of $\CI$.

\begin{lemma}\label{L:closedset1}
Suppose that $M$ is an $A$-module and $P$ is an indecomposable 
projective $A$-module.  Let $s$ be a nonnegative integer.
Let $\CW$ be the subset of $\CV_n$ consisting 
of all representations $\sigma$
of $A$ having the property that $M \otimes L_{\sigma}$ 
has no submodule 
isomorphic $P^s$, where $L_{\sigma}$ is 
the $A$-module affording $\sigma$.
Then $\CW$ is a closed set in $\CV_n$. 
\end{lemma}

\begin{proof}
By the previous lemma, there is an element 
$u \in A$ such that, if $t$ is the
rank of the matrix of $u$ acting on $P$, 
then $P^s$ is isomorphic to a submodule
of $M \otimes L_{\sigma}$  if and only if 
the rank of the matrix $\CM_u$ of $u$ on 
$M \otimes L_{\sigma}$  is at least $st$. 
Now we observe that $u$ is a polynomial
in the (noncommuting) generators of $A$ and 
hence the entries of the matrix 
of $u$ on $L_{\sigma}$ are polynomials in 
the entries of the matrices of the generators
of $A$ on $L_{\sigma}$. If we fix a representation 
of $M$, then the entries of 
the matrix of $u$ on $M$ are elements of 
the base field $k$. It follows that the
entries of the matrix $\CM_u$ of $u$ on  
$M \otimes L_{\sigma}$ are all polynomials in 
the variables of the ring $R = k[x_{rs}^i]$. 
Likewise, the determinant of any 
$st \times st$ submatrix of $\CM_u$ is a 
polynomial in the variables of $R$.
As a consequence, the condition that every 
such determinant is zero (which is
the same as saying that the rank of $\CM_u$ 
is less than $st$) defines a closed
set in $\CV_n$. 
\end{proof}

\begin{lemma}\label{L:closedset2} 
Suppose that $M$ is an endotrivial $A$-module 
of dimension $n$. Let $\CW$ be the 
set of all representations $\sigma$ in 
$\CV_n$ such that the module $L_{\sigma}$
afforded by $\sigma$ is not isomorphic to 
$M \otimes \chi$ for any one dimensional
$A$-module $\chi$. Then $\CW$ is a closed set in $\CV_{n}$.
\end{lemma}

\begin{proof}
Suppose that $N$ is a one dimensional $A$-module.
Because $M$ is endotrivial, we have that 
$N\otimes M$ is endotrivial. Thus 
\[
(N \otimes M) \otimes (N^* \otimes M^*) \cong k \oplus 
\sum_{i=1}^r P_i^{n_i}
\]
where $P_1, \dots, P_r$ are the indecomposable 
projective $A$-modules and $n_1, \dots, n_r$ 
are nonnegative integers. 
For each $i$, let $\CW_i$ be the set of all 
$\sigma \in \CV_n$ with the property
that $L_{\sigma} \otimes N^* \otimes M^*$ 
does not contain a submodule isomorphic to 
$P_i^{n_i}$.  The sets $\CW_i$ are closed 
by Lemma~\ref{L:closedset1}. Hence, 
the set $\CU_N = \CW_1 \cup \dots \cup \CW_r$ 
is also closed. 
If $\sigma$ is not in $\CU_N$, then 
\[
L_{\sigma} \otimes N^* \otimes M^* \quad \cong \quad U \ \oplus \ 
\sum_{i=1}^r P_i^{n_i}
\]
for some $A$-module $U$. 
But the dimension of $U$ must be one because 
$\dim L_{\sigma}\otimes N^{*}\otimes M^{*}=\dim (N\otimes M)\otimes (N^{*}\otimes M^{*})$. Therefore, 
\[
L_{\sigma} \otimes N^* \otimes M^* \otimes U^*\quad \cong 
\quad k \ \oplus \ \sum_{i=1}^r U^* \otimes P_i^{n_i}
\]
and hence, $L_{\sigma} \cong U \otimes N \otimes M$. Now we claim that 
$\CW = \cap_N \CU_N$ where $N$ runs through the one dimensional 
$A$-modules. So $\CW$ is closed. 
\end{proof}

At this point we are ready to prove our main theorem.

\begin{thm} \label{mainthm}
For any natural number $n$, there is only 
a finite number of isomorphism classes of 
endotrivial modules of dimension $n$. 
\end{thm}

\begin{proof}
Suppose that $M$ is an indecomposable endotrivial
module of dimension $n$. Let $\CU$ be 
the subset of $\CV_n$ consisting of 
representations $\sigma$ with the property that 
the underlying module $L_{\sigma}$ is 
isomorphic to $N \otimes M$ for some 
$A$-module $N$ of dimension one. Note 
that $A$ has only finitely many isomorphism
classes of dimension one, and hence 
there are only finitely many isomorphism classes
of modules represented in $\CU$.

By Lemma~\ref{L:closedset2}, $\CU$ is an open set 
in $\CV_n$. Hence, $\overline{\CU}$, the
closure of $\CU$ is a union of components 
in $\CV_n$. Therefore, the theorem is 
proved with the observation that $\CV_n$ 
has only finitely many components. 
\end{proof}

\section{Liftings and Stability} 

Let $S$ be a finite group scheme which is 
a normal subgroup scheme in a group scheme $H$. 
In this section we will describe different 
notions of when an $S$-module has a structure that 
extends to $H$. 

We say that an $S$-module $M$ {\em lifts} 
to $H$ if $M$ has an $H$-module structure whose 
restriction to $S$ agrees with the (original) 
$S$-module structure. This is the strongest form 
of ``lifting''. The weakest form of lifting 
is the notion of $H$-stable. Let $M$ be a $S$-module. 
For $h\in H$, one can consider the 
twisted module $M^{h}$ which is a 
$S^{\prime}=h^{-1}Sh$-module (cf. \cite[I. 2.15]{Jan}).  
In particular if $h$ normalizes $S$ 
then the twisted module $M^{h}$ is an $S$-module. 
An $S$-module $M$ is called {\em $H$-stable} 
if and only if $M^{h}\cong M$. 
If the $S$-module $M$ lifts to $H$-module, then 
$M$ is $H$-stable. The converse statement is 
not true as we will see in Section 6 (cf. \cite[4.2.1]{PS}). 

Following \cite[2.2.2, 2.2.3]{PS} we recall 
the notions of numerical and tensor stablity 
defined by Parshall and Scott. 

\begin{defi} An $S$-module $M$ is 
{\em numerically $H$-stable} if there exists 
an $H$-module $Z$ such that $Z|_{S}\cong M^{\oplus n}$. 
\end{defi} 

\begin{defi} An $S$-module $M$ is {\em tensor $H$-stable} 
if there exists a finite-dimensional $H/S$-module $Y$ such that 
$M\otimes Y$ is an $H$-module whose restriction 
to $S$ coincides with $M\otimes (Y|_{S})$. 
\end{defi} 

Tensor $H$-stability is equivalent to numerically $H$-stable (cf. \cite[2.2.3]{PS}). It is clear that if $M$ lifts to $H$ then 
$M$ is tensor $H$-stable and numerically $H$-stable. 
Furthermore, tensor $H$-stable and numerically 
$H$-stable imply $H$-stable. 

Next we introduce a new concept of lifting which 
will be relevant for our study of endotrivial modules. 

\begin{defi} An $S$-module $M$ {\em stably lifts to $H$} 
if there exists an $H$-module $K$ such that 
$K|_{S}\cong M\oplus P$  where $P$ is a projective $S$-module. 
\end{defi} 

Observe that in the definition of stable lifting 
to $H$, the $S$-modules $K$ and $M$ 
represent the same object in $\text{StMod}(S)$, 
the stable category of all 
$S$-modules. 
If $M$ lifts $H$ then $M$ stably lifts to $H$. 
Also, if $M$ is non-projective as an $S$-module 
and stably lifts to $H$ then by using the 
Krull-Schmidt theorem and the fact that twists 
of projective $S$-modules are projective, it 
follows that $M$ is $H$-stable. 
 
\section{Applications} 

Suppose that  $G$ is  a semisimple, simply connected 
algebraic group, defined and split over the finite 
field ${\mathbb F}_p$ with $p$ elements for a
prime $p$. Let $k$ be the algebraic closure of ${\mathbb F}_p$.  Let
$\Phi$ be a root system associated to $G$ with respect to a maximal
split torus $T$. Let $\Phi^{+}$ (resp. $\Phi^{-}$) be the set of positive
(resp. negative) roots and $\Delta$ be a base consisting of simple roots. 
Let $B$ be a Borel subgroup containing $T$ corresponding 
to the negative roots and let $U$ denote the unipotent radical of $B$. 
More generally, if $J\subset \Delta$, let $L_{J}$ be the Levi subgroup generated 
by the root subgroups with roots in $\Delta$, $P_{J}$ the associated (negative) parabolic 
subgroup and $U_{J}$ its unipotent radical such that $P_{J}=L_{J}\ltimes U_{J}$. 

Let $H$ be an affine algebraic group scheme 
over $k$ and let $H_{r}=\text{ker }F^{r}$. 
Here $F:H\rightarrow H^{(1)}$ is the Frobenius map
and $F^{r}$ is the $r^{th}$  iteration of the Frobenius map.
We note that there is a categorical equivalence between 
modules for the restricted $p$-Lie algebra
$\text{Lie}(H)$ of $H$  and $H_{1}$-modules.
For each value of $r$, the group algebra 
$kH_r$ is the distribution 
algebra $\Dist(H_r)$ (cf. \cite{Jan}). In 
general, for the rest of this 
paper, we use $\Dist(H_{r})$ to denote 
the group algebra of $H_r$.

For any group scheme $H$, let $\text{mod}(H)$ 
be the category of
finite dimensional rational $H$-modules. This 
construction can be applied 
when $H=G$, $B$, $P_{J}$, $L_{J}$, $U$, $U_{J}$ 
and $T$. Note that the use 
of $T$ (maximal torus) and $T(-)$ 
(endotrivial group) will be clear 
from the context. Let $X:=X(T)$ be 
the integral weight lattice obtained from $\Phi$.
The set $X$ has a partial ordering: 
if $\lambda,\mu\in X$, then
$\lambda\geq \mu$ if and only if 
$\lambda - \mu\in \sum_{\alpha\in
\Pi}\mathbb{N}\alpha$.  

Let $\alpha^{\vee}=2\alpha/\langle\alpha,\alpha\rangle$ 
is the coroot
corresponding to $\alpha\in \Phi$. The set 
of dominant integral weights is defined by
$$
X_{+}:=X(T)_{+}=\{\lambda\in X(T):
\ 0\leq \langle\lambda,\alpha^{\vee}\rangle\
\text{for all $\alpha \in \Delta$} \}.
$$
Furthermore, the set of $p^{r}$-restricted 
weights is
$$
X_{r}(T)=\{\lambda\in X:\
0\leq\langle\lambda,\alpha^{\vee}\rangle<p^{r}\,\,
\text{for all $\alpha\in
\Delta$}\}.
$$
Let $X(T_{r})$ be the set of characters 
of $T_{r}$ which can 
be identified with the set of one 
dimensional simple modules for $T_{r}$. 

For a reductive algebraic group $G$, the 
simple modules are labelled $L(\lambda)$ 
and the induced modules are 
$H^{0}(\lambda)=\text{ind}_{B}^{G} \lambda$, 
where $\lambda\in X(T)_{+}$. 
The Weyl module $V(\lambda)$ is defined as 
$V(\lambda)=H^{0}(-w_{0}\lambda)^{*}$. Let $T(\lambda)$ be 
the indecomposable tilting module with 
highest weight $\lambda$. 

We can now apply the Finiteness Theorem 
to demonstrate, under mild assumptions on $H$, that 
every endotrivial module is $H$-stable. 
In the following sections we show
that the problem of lifting of endotrivial 
modules is rather subtle. 

\begin{thm} Let $H$ be a connected affine algebraic 
group scheme and $S$ be 
a finite group scheme which is a normal 
subgroup scheme of $H$. If $M$ is an 
endotrivial $S$-module then $M$ 
is $H$-stable.  
\end{thm} 

\begin{proof} Consider the closed subgroup 
$A=\{h\in H:\ M^{h}\cong M\}$ in $H$. According to 
the finiteness theorem, this must have finite 
index in $H$ because there are only finitely many 
endotrivial $S$-modules of any fixed dimension. 
Therefore, $A$ must 
contain the connected component 
of $H$. Because $H$ is connected, we have 
that $A=H$, which proves the theorem. 
\end{proof} 

\begin{cor} \label{C:twistiso} Let $H=G$, 
$B$, $P_{J}$, $L_{J}$, $U$, $U_{J}$ and $T$ as above, and 
$H_{r}$ be the $r$th Frobenius kernel. 
If $M$ is an endotrivial $H_{r}$-module, 
then $M$ is $H$-stable. 
\end{cor} 

Another application of the Finiteness Theorem 
involves proving that the restriction of an 
endotrivial $G_{r}$-modules 
to conjugate unipotent radicals of Borel subgroups 
produces syzygies of the same degree. 

\begin{thm} Let $U$ and $U^{\prime}$ be the 
unipotent radicals of the Borel subgroups 
$B$ and $B^{\prime}$. If $M$ is an endotrivial 
$G_{r}$-module with 
$$
M|_{U_{r}}\cong \Omega^{n_{1}}_{U_{r}}(k)\oplus (proj)
$$
and 
$$
M|_{U^{\prime}_{r}}\cong 
\Omega^{n_{2}}_{U^{\prime}_{r}}(k)\oplus (proj)
$$
then $n_{1}=n_{2}$. 
\end{thm} 

\begin{proof} The Borel subgroups $B$ 
and $B^{\prime}$ are conjugate by some $w\in W$, and 
$B^{\prime}=B^{w}=w(B)w^{-1}$. According 
to Corollary~\ref{C:twistiso}, we have 
$M\cong M^{w}$ as $G_{r}$-modules. Under 
this isomorphism $U_{r}$ is isomorphic to $U_{r}^{\prime}$ and 
$M|_{U_{r}}$ can be identified with 
$M|_{U^{\prime}_{r}}=M^{w}$. The result 
now follows by applying these isomorphisms. 
\end{proof} 

\section{Lifting Endotrivial Modules}

Let $M$ be an $S$-module and $\Omega^{n}_{S}(M)$ 
($n=0,1,2,\dots $) be the $n^{th}$ syzygy of $M$ obtained by 
taking a projective resolution of $M$. We 
will assume that $\Omega^{n}_{S}(M)$ has 
no projective $S$-summands. 
By taking an injective resolution of $M$ 
we can define $\Omega^{n}_{S}(M)$ 
for $n$ negative. Note that if $M$ 
is endotrivial over $S$ then $\Omega_{S}^{n}(M)$ 
is an endotrivial module 
for all $n\in {\mathbb Z}$. The following result 
provides conditions on when there are stable 
liftings for the syzygies of an $H$-module $M$ 
(when considered as an $S$-module). 

\begin{prop} \label{P:stableiso} Let $S$ be 
a finite group scheme which is a normal 
subgroup scheme of $H$. Suppose there exists 
a projective $S$-module $P$ which lifts 
to an $H$-module. Furthermore, suppose that there exists 
a surjective $H$-map $P\rightarrow k$.  If $M$ is a 
finite-dimensional $H$-module then 
$\Omega^{n}_{S}(M)$ stably lifts to $H$ for each 
$n\in {\mathbb Z}$ . 
\end{prop} 

\begin{proof} Consider the surjective $H$-module 
homomorphism $P\rightarrow k$. We will prove that 
$\Omega^{n}_{S}(M)$ stably lifts by induction on 
$n\in {\mathbb N}$. For $n\leq 0$ a 
similar inductive argument can be used. 

For $n=0$, $\Omega_{S}^{0}(M)=M$ which is 
an $H$-module so we can set $K_{0}=M$. Now 
assume that $K_{n}|_{S}\cong 
\Omega^{n}_{S}(M)\oplus Q_{n}$ where 
$K_{n}$ is an $H$-module $K_{n}$ and $Q_{n}$ is 
a projective $S$-module. Now define 
$K_{n+1}$ as the kernel in the short 
exact sequence obtained by tensoring 
the complex $P\rightarrow k$ by $K_{n}$: 
$$
0\rightarrow K_{n+1}\rightarrow P\otimes K_{n} 
\rightarrow K_{n} \rightarrow 0.
$$ 
Then $K_{n+1}$ is an $H$-module with 
$K_{n+1}|_{S}\cong \Omega^{n+1}_{S}(M)\oplus Q_{n+1}$ 
for some projective $S$-module $Q_{n+1}$. 
\end{proof}

Let $G$ be a reductive group with subgroups 
$P$, $B$ and $U$as before. The existence of 
the Steinberg representation $St_{r}$ can be 
used to prove that every projective $G_{r}$ 
(resp. $P_{r}$, $B_{r}$, $U_{r}$) module stably 
lifts to $G$ (resp. $P$, $B$, $U$). 
It is only known for $p\geq 2(h-1)$ that 
projective $G_{r}$-modules lift to $G$, 
but there is strong evidence this holds for all 
$p$. We can now prove that for reductive 
groups and their associated Lie type subgroups that 
the syzygies of the trivial module lift stably. 
The proof also utilizes the existence of the 
Steinberg representation.

\begin{thm}\label{T:reductive} Let $H=G$ 
(resp. $P$, $B$, $U$) and $S=G_{r}$ 
(resp. $P_{r}$, $B_{r}$, $U_{r}$). For each 
$n\in {\mathbb Z}$, $\Omega_{S}^{n}(k)$ 
stably lifts to $H$. 
\end{thm} 

\begin{proof} It suffices to prove the theorem in the 
case that $H=G$ and $S=G_{r}$. The other cases will follow 
by restriction. Let $\text{St}_{r}=L((p^{r}-1)\rho)$ 
be the Steinberg module, and set $P:=
\text{St}_{r}\otimes L((p^{r}-1)\rho)$. Then there 
exists a surjective $G$-module homomorphism $P\rightarrow k$ 
\cite[II 10.15 Lemma]{Jan}. The result now 
follows by Proposition~\ref{P:stableiso}
\end{proof}

We note that it is not trivial to prove the fact 
that the left $U_{r}$-module structure on 
$\text{Dist}(U_{r})$ lifts to $H=U$. 
The conjugation action of $U_{r}$ on 
$\text{Dist}(U_{r})$ lifts to $U$ and there 
exists a $U$-module map $\text{Dist}(U_{r})\rightarrow k$ 
under the conjugation action. However, 
the module $\text{Dist}(U_{r})$ is not 
a projective module under this action 
(i.e., the conjugation action does not 
lift the left action of $\text{Dist}(U_{r})$ 
on itself). 

\begin{cor} Let $H=B$ (resp. $U$), and 
$S=B_{r}$ (resp. $U_{r}$). Then every endotrivial 
$S$-module lifts stably 
to an $H$-module. 
\end{cor} 

\begin{proof} We first consider the case 
that $H=B$ and $S=B_{r}$. According to 
\cite[Theorem 6.1, 6.2]{CN}, 
$T(B_{r})\cong X(T_{r})\times T(U_{r})$. 
The one dimensional $B_{r}$ endotrivial 
modules corresponding elements 
of $X(T_{r})$ are all $B$-modules. 
Therefore, it suffices to prove the 
statement when $H=U$ and $S=U_{r}$. 

Assume that $\Phi$ is not $A_{2}$ in the case that
$p=2$ and $r=1$. Then any endotrivial 
$B_{r}$-module is isomorphic to 
$\Omega^{n}_{B_{r}}(\lambda)$ for some 
$\lambda\in X_{r}(T)$. Since $\lambda$ 
lifts to a $B$-module, by Theorem~\ref{T:reductive}, 
$\Omega^{n}_{B_{r}}(\lambda)$ stably lifts to $B$. 

In the case when $\Phi$ is of type $A_{2}$ 
the endotrivial group $T(B_{1})$ is generated 
by $\Omega^{1}_{B_{1}}(\lambda)$ and 
the simple three dimensional $G$-module 
$L(\omega_{1})$ considered as $B_{1}$-module 
by restriction. Since $L(\omega_{1})$ 
is a $B$-module all of its syzygies 
$\Omega^{n}_{B_{1}}(L(\omega_{1}))$ stably 
lift to $B$ by Proposition~\ref{P:stableiso}. 
\end{proof} 

When $G$ is a reductive algebraic group 
scheme we can state a relationship between 
a $G_{r}$-module lifting stably to $G$ and 
tensor stability as a direct application of 
\cite[Theorem 1.1]{PS}. This seems to indicate 
that stably lifting is a stronger form of 
lifting that tensor stability. 

\begin{prop} Let $G$ be reductive and let $M$ 
be a $G_{r}$-module which lifts stably to $G$. Suppose that $N$ 
is a $G$-module such that $N|_{G_{r}}=M\oplus P$ 
where $P$ is a projective $G_{r}$-module (i.e., a stable 
lifting of $M$). If $\operatorname{soc}_{G_{r}}M$ 
is a $G$-submodule of $N$ then $M$ is tensor $G_{r}$-stable. 
\end{prop} 

In the next theorem we give a condition on 
the quotient $H/S$ which insures that we 
can lift syzygies. 

\begin{thm} \label{T:lift} Let $S$ be a 
finite group scheme which is a normal subgroup scheme of $H$. 
Assume that 
\begin{itemize} 
\item[(i)] If $L$ is a simple $H$-module, 
then $L|_{S}$ is a simple $S$-module, and 
all simple $S$-modules lift to $L$. 
\item[(ii)] For any simple $H$-module $L$ 
there exists a $H$-module $Q(L)$ such that 
$Q(L)|_{S}$ is the projective cover $L|_{S}$. 
\item[(iii)] All finite-dimensional modules 
for $H/S$ are completely reducible. 
\end{itemize} 
Let $M$ be a finite-dimensional $H$-module. 
Assume that the projective cover $P(M)$ of $M$ as 
an $S$-module lifts to an $H$-module and there exists 
a surjective $H$-homomorphism $P(M)\rightarrow M$.  
Then $\Omega^{n}_{S}(M)$ lifts to an $H$-module 
for all $n\in {\mathbb Z}$. 
\end{thm} 

\begin{proof} We begin with an observation about the 
cohomology. For any $H$-module $N$, there exists a 
Lyndon-Hochschild-Serre (LHS) spectral sequence: 
$$
E_{2}^{i,j}:\text{Ext}^{i}_{H/S}(k,\text{Ext}^{j}_{S}(k,N))
\Rightarrow \text{Ext}^{i+j}_{H}(k,N).
$$ 
Condition (iii) implies that this spectral sequence 
collapses and hence, the 
restriction map $\text{Ext}^{j}_{H}(k,N)\rightarrow 
\text{Ext}^{j}_{S}(k,N)^{H/S}$ is an 
isomorphism for all $j\geq 0$. 

If suffices to assume that $n\geq 0$ and 
prove the theorem by induction. The case 
when $n$ is negative can be handled by 
using a dual argument. For $n=0$, we have that
$\Omega^{0}_{S}(M)\cong M$ and for $n=1$, 
$\Omega^{1}_{S}(M)\cong \text{ker}(P(M)\rightarrow M)$. 
So these modules lift to $H$. 

Suppose that $\Omega^{n}_{S}(M)$ lifts to $H$. 
The $S$-submodule $\text{Rad}_{S}\Omega^{n}_{S}(M)$ 
is an $H$-submodule of $\Omega^{n}_{S}(M)$ so 
there exists a surjective  $H$-module map 
$$
\pi: \Omega^{n}_{S}(M)\rightarrow \Omega^{n}_{S}(M)/ 
\text{Rad}_{S}\Omega^{n}_{S}(M).
$$
The quotient module  
$\Omega^{n}_{S}(M)/ \text{Rad}_{S}\Omega^{n}_{S}(M)$ 
is completely reducible 
as an $S$-module. Furthermore, it must be 
completely reducible as an $H$-module. For if there exists 
a non-trivial extension of simple 
$H$-modules which lives as an $H$-submodule 
in $\Omega^{n}_{S}(M)/ \text{Rad}_{S}\Omega^{n}_{S}(M)$,
then by condition (i), these simple modules 
remain simple upon restriction to $S$ 
and this extension must split over $S$ 
(by complete reducibility of 
the quotient module). Then by the first 
observation above, the original extension 
over $H$ must split. 

By condition (ii), there exists an $H$-module 
$Q$ whose restriction to $S$ is the projective 
cover of $\Omega^{n}_{S}(M)/ \text{Rad}_{S}\Omega^{n}_{S}(M)$ 
with a surjective $H$-module map 
$\gamma:Q\rightarrow 
\Omega^{n}_{S}(M)/ \text{Rad}_{S}\Omega^{n}_{S}(M)$.  
We have a 
short exact sequence of $H$-modules: 
$$
0\rightarrow \text{Rad}_{S}\Omega^{n}_{S}(M)
\rightarrow \Omega^{n}_{S}(M)\rightarrow 
\Omega^{n}_{S}(M)/ \text{Rad}_{S}\Omega^{n}_{S}(M) \rightarrow 0.
$$ 
Observe that $\text{Ext}^{1}_{H}(Q,\text{Rad}_{S}\Omega^{n}_{S}(M))=\text{Ext}^{1}_{S}(Q, \text{Rad}_{S}\Omega^{n}_{S}(M))^{H/S}=0$. 
Therefore, in the long exact sequence 
in cohomology the map 
$$
\text{Hom}_{H}(Q,\Omega^{n}_{S}(M))\rightarrow  \text{Hom}_{H}(Q,\Omega^{n}_{S}(M)/ \text{Rad}_{S}\Omega^{n}_{S}(M))
$$ 
is surjective and we can find and 
$H$-module map $\delta:Q\rightarrow \Omega^{n}_{S}(M)$ 
such that  $\pi\circ \delta=\gamma$ and 
$\Omega^{n+1}_{S}(M)=\text{ker }\delta$. 
Consequently, $\Omega^{n+1}_{S}(M)$ is an $H$-module.  
\end{proof} 

In the case that $H=G_{r}T$ (resp. $P_{r}T$, $B_{r}T$) 
and $S=G_{r}$ (resp. $P_{r}$, $B_{r}$) 
conditions (i)-(iii) of the 
preceding theorem can be verified 
(cf. \cite[Chapter 9]{Jan}).

\begin{cor} \label{C:G_rTlift} 
Let $H=G_{r}T$ (resp. $P_{r}T$, $B_{r}T$) 
and $S=G_{r}$ (resp. $P_{r}$, $B_{r}$). If $M$ is 
an $H$-module then for each $n\in {\mathbb Z}$, 
$\Omega_{S}^{n}(M)$ is an $H$-module. 
\end{cor} 

The conditions (i)-(iii) do not hold for 
when $H=G$ and $S=G_{r}$. Nonetheless, we can prove 
by a direct calculation that all endotrivial 
$G_{1}$-modules for $G=SL_{2}$ lift to $G$. 

\begin{thm} Let $G=SL_{2}$. Then every 
endotrivial $G_{1}$-module lifts to $G$. 
\end{thm}

\begin{proof} The category of $G_{1}$-modules 
has tame representation type and the indecomposable 
modules have been determined (cf. \cite[Section 3]{CNP}). 
The modules of complexity two, which include all 
endotrivial modules, lift to $G$. 

One can also verify this by using the classification 
of endotrivial modules given in \cite{CN}. 
The endotrivial group is 
${\mathbb Z}\oplus {\mathbb Z}_{2}$ and all 
endotrivial modules 
are of the form $\Omega^{n}(M)$ 
where $n\in {\mathbb Z}$ and $M=k$ or 
$M = L(p-2)$. 

In either case $M \cong M^*$ where $M^*$ 
is the $k$-dual of $M$. 
Hence  $\Omega^{n}(M)\cong \Omega^{-n}(M)^{*}$, 
and  without loss of generality, we may 
assume that $n\geq 0$. The minimal projective 
resolution $P_{\bullet} \rightarrow k$ can be constructed 
explicitly. All the terms are tilting modules: 
$$
P_{n}\cong \begin{cases} T((\frac{n}{2}+1)2(p-1)) &  \text{if $n$ is even}, \\
T((\frac{n+1}{2})2p) & \text{if $n$ is odd}. 
\end{cases}
$$ 
Moreover, the syzygies are Weyl modules 
$$
\Omega^{n}(k)\cong 
\begin{cases} V(np) &  \text{if $n$ is even}, \\
V((\frac{n+1}{2})2(p-2)) & \text{if $n$ is odd}.
\end{cases}
$$ 
Using a similar construction, the minimal projective resolution 
$\widehat{P}_{\bullet}\rightarrow L(p-2)$ consists of 
tilting modules and the syzygies are also Weyl modules. 
$$
\widehat{P}_{n}\cong \begin{cases} 
T((n+1)p) & \text{if $n$ is even}, \\
T((n+2)p-2) & \text{if $n$ is odd}, 
\end{cases}
$$ 
$$
\Omega^{n}(L(p-2))\cong 
\begin{cases} V((n+1)p-2 ) &  \text{if $n$ is even}, \\
V(np) & \text{if $n$ is odd}.
\end{cases}
$$ 
\end{proof}

\section{An example where $\Omega^{2}_{S}(k)$ does not lift to $H$} 

In this section we show that syzygies of the trivial module 
do not, in general, lift to $H$-modules even in cases 
where all the projective indecomposable $S$-modules 
lift to $H$. Suppose that  
$G=SL_{3}$ and that $p=2$. Let 
$S=B_{1}$, $H=B$. The restricted $p$-Lie algebra ${\mathfrak u}$ of 
the unipotent radical of $G$, has the same representation theory as 
the infinitesimal unipotent subgroup $U_1 \subseteq B_1$.
In this context we prove the following. 

\begin{prop}\label{omeganolift}
The second syzygy $\Omega^{2}(k):=\Omega^{2}_{B_{1}}(k)$ 
stably lifts to a $B$-module, but does not lift to a $B$-module. 
\end{prop}

\begin{proof}
The first statement follows from Theorem \ref{T:reductive}. 
We suppose that $\Omega^{2}(k)$ has a $B$-structure and prove that
this leads to a contradiction. 

Let $V:=L(\omega_{1})$ be the three dimensional natural representation for $G$ and label the simple roots 
$\Delta=\{\alpha_{1},\alpha_{2}\}$. We will consider the 
restriction of $V$ to $B$ and $B_{1}$. Set 
$$N_{1}\cong V\otimes (-2\alpha_{1}-\alpha_{2}-\omega_{1})$$ 
and 
$$N_{2}\cong V\otimes (-\alpha_{1}-2\alpha_{2}-\omega_{1}).$$ 

We can represent $\Omega^{2}(k)$ diagrammatically as in Figure 1. 
A node ($\lambda$) is a one-dimensional 
$B_i$-submodule with highest
weight $\lambda$.  The arrow indicate the 
action of the simple
root-subspace vectors in ${\mathfrak u}:=\text{Lie }U$, 
and the extensions between the 
simple one-dimensional $B_{1}$-modules. That is, an arrow that goes
down and left represents multiplication by $u_{\alpha_1}$, while 
an arrow going down and right is multiplication by $u_{\alpha_2}$.

\begin{figure}\label{Omega2diagram} 
\begin{center}
\setlength{\unitlength}{.5cm}
\begin{picture}(20,8)
\put(8,3.5){($-2\alpha_{1}-2\alpha_{2}$)}
\put(4,0.0){($-3\alpha_{1}-2\alpha_{2}$)}
\put(12,0.0){($-2\alpha_{1}-3\alpha_{2}$)}
\put(2,2.0){($-3\alpha_{1}-\alpha_{2}$)}
\put(14,2.0){($-\alpha_{1}-3\alpha_{2}$)}
\put(0,4.0){($-2\alpha_{1}-\alpha_{2}$)}
\put(16,4.0){($-\alpha_{1}-2\alpha_{2}$)}
\put(-0.5,6.0){($-2\alpha_{1}$)}
\put(18.2,6.0){($-2\alpha_{2}$)}
\put(1.0,5.7){\vector(1,-1){0.8}}
\put(2.8,3.7){\vector(1,-1){0.8}}
\put(4.6,1.6){\vector(1,-1){0.8}}
\put(9.5,3.0){\vector(-1,-1){2.0}}
\put(10.5,3.0){\vector(1,-1){2.0}}
\put(19,5.7){\vector(-1,-1){0.8}}
\put(17,3.7){\vector(-1,-1){0.8}}
\put(15.0,1.6){\vector(-1,-1){0.8}}
\end{picture}
\caption{}
\end{center}
\end{figure}

By analyzing the structure of $\Omega^{2}(k)$ we can conclude that 
$$\Omega^{2}(k)/\text{Rad}_{B_{1}}\Omega^{2}(k)\cong {\mathfrak u}^{(1)}$$
and 
$$\text{Rad}_{B_{1}}\Omega^{2}(k)\cong N_{1}\oplus N_{2}$$ 
as $B$-modules. 

The module $\Omega^{2}(k)$ is indecomposable over $B_{1}$. Therefore, if $\Omega^{2}(k)$ has a $B$-structure then it 
is indecomposable over $B$ and represents a non-trivial extension class in 
$$\text{Ext}^{1}_{B}(\Omega^{2}(k)/\text{Rad}_{B_{1}}\Omega^{2}(k), \text{Rad}_{B_{1}}\Omega^{2}(k)).$$
This implies that $\text{Ext}^{1}_{B}({\mathfrak u}^{(1)}, N_{j})\neq 0$ for $j=1$ or $2$. 

Our task is to show by a cohomological calculation that 
$\text{Ext}^{1}_{B}({\mathfrak u}^{(1)}, N_{j})=0$ for $j=1$ and $2$. 
This provides a contradiction to the assumption 
that $\Omega^{2}(k)$ has a compatible $B$-structure. By symmetry we can 
simply look at the case that $j=1$. Apply the LHS spectral sequence 
\begin{equation}\label{E1}
E_{2}^{i,j}=
\text{Ext}^{i}_{B/B_{1}}(k,\text{Ext}^{j}_{B_{1}}({\mathfrak u}^{(1)},N_{1})
\Rightarrow \text{Ext}^{i+j}_{B}({\mathfrak u}^{(1)},N_{1}).
\end{equation}
Note that $\text{Hom}_{B_{1}}(k,N_{1})=0$ so the five term exact sequence (associated to this spectral sequence) yields: 
\begin{equation}\label{E1-term}
E_{1}=\text{Ext}^{1}_{B}({\mathfrak u}^{(1)},N_{1})\cong \text{Hom}_{B/B_{1}}({\mathfrak u}^{(1)},\text{Ext}_{B_{1}}^{1}(k,N_{1})).
\end{equation}

We can utilize the techniques in  \cite[Lemma 3.1.1, Theorem 3.2.1]{UGA} to compute the $B/B_{1}$-socle of $\text{Ext}_{B_{1}}^{1}(k,N_{1})$. Observe 
that as a $B/B_{1}$-module: 
\begin{eqnarray*} 
\text{Ext}_{B_{1}}^{1}(k,N_{1})&\cong& \text{Ext}^{1}_{B_{1}}(L(\omega_{1})^{*},-2\alpha_{1}-\alpha_{2}-\omega_{1})\\
&\cong& \text{Ext}^{1}_{B_{1}}(L(\omega_{2}),-2\alpha_{1}-\alpha_{2}-\omega_{1}).
\end{eqnarray*} 
Let $-p\nu$ be a simple module in the socle where $\nu\in X$. Recall that $X$ is the set of weights and $X_{+}$ is the set of dominant weights. 
Then 
\begin{eqnarray*}
\text{Hom}_{B/B_{1}}(-p\nu,\text{Ext}^{1}_{B_{1}}(L(\lambda),\mu))
&\cong& \text{Hom}_{B/B_{1}}(k,\text{Ext}^{1}_{B_{1}}(L(\lambda),\mu)\otimes p\nu)\\
&\cong&  \text{Hom}_{B/B_{1}}(k,\text{Ext}^{1}_{B_{1}}(L(\lambda),\mu + p\nu). 
\end{eqnarray*} 

Set $\lambda=\omega_{2}$ and $\mu=2\alpha_{1}-\alpha_{2}-\omega_{1}$. Consider the LHS spectral sequence 
$$E_{2}^{i,j}=\text{Ext}^{i}_{B/B_{1}}(k,\text{Ext}^{j}_{B_{1}}(L(\lambda),\mu+p\nu))\Rightarrow \text{Ext}^{i+j}_{B}(L(\lambda),\mu+p\nu).$$ 
So $\text{Hom}_{B_{1}}(L(\lambda),\mu+p\nu)=0$, 
because $\lambda-\mu \notin pX$. The associated 
five term exact sequence yields an isomorphism given as  
$$
E_{2}^{0,1}= 
\text{Hom}_{B/B_{1}}(k,\text{Ext}^{1}_{B_{1}}(L(\lambda),\mu + p\nu))
\cong \text{Ext}^{1}_{B}(L(\lambda),\mu+p\nu).
$$ 

There exists another spectral sequence 
$$E_{2}^{i,j}=\text{Ext}^{i}_{G}(L(\lambda),R^{j}\text{ind}_{B}^{G}(\mu+p\nu))\Rightarrow \text{Ext}^{i+j}_{B}(L(\lambda),\mu+p\nu).$$
We have two cases to consider. Suppose $\mu+p\nu\in X_{+}$.  
Then by Kempf's vanishing theorem,
this spectral sequence collapses and 
we have that 
$$\text{Ext}^{1}_{B}(L(\lambda),\mu+p\nu)\cong \text{Ext}^{1}_{G}(L(\omega_{2}),H^{0}(\mu+p\nu))=
\text{Ext}^{1}_{G}(V(\omega_{2}),H^{0}(\mu+p\nu))=0.$$  
On the other hand, if $\mu+p\nu\notin X_{+}$, 
then the five term exact sequence yields 
$$
\text{Ext}^{1}_{B}(L(\lambda),\mu+p\nu)\cong 
\text{Hom}_{G}(L(\omega_{2}),R^{1}\text{ind}_{B}^{G}(\mu+p\nu)).
$$ 
By results of Andersen \cite[Proposition 2.3]{And}, we have that
$\mu+p\nu=s_{\alpha}\cdot \omega_{2}$ where $\alpha\in \Delta$, and 
$s_{\alpha}$ is a simple reflection. A direct computation shows that 
$$\mu-s_{\alpha_{1}}\cdot \omega_{2}=-2(\alpha_{1}+\alpha_{2})$$
and 
$$\mu-s_{\alpha_{2}}\cdot \omega_{2}=-3\alpha_{1}.$$ 
The second condition can not be satisfied because 
$-3\alpha_{1}\notin pX$. Therefore, the $B/B_{1}$ socle of 
$\text{Ext}_{B_{1}}^{1}(k,N_{1})$ is one-dimensional 
and is equal to $-2(\alpha_{1}+\alpha_{2})$. 

In addition, $\text{Ext}^{1}_{B_{1}}(k,N_{1})$ is a subquotient of $\text{Hom}_{B_{1}}(P(k),N_{1})$ (where $P(k)$ is the projective cover of 
$k$ as $B_{1}$-module), and it has dimension at most two. 
Furthermore, the $T$-weights of ${\mathfrak u}^{(1)}$ are distinct and 
the $B/B_{1}$-socle of ${\mathfrak u}^{(1)}$ is $-2(\alpha_{1}+\alpha_{2})$, so the image of any any non-zero map in 
$\text{Hom}_{B/B_{1}}({\mathfrak u}^{(1)},\text{Ext}_{B_{1}}^{1}(k,N_{1}))$ is three-dimensional. We can 
now conclude that $\text{Hom}_{B/B_{1}}({\mathfrak u}^{(1)},\text{Ext}_{B_{1}}^{1}(k,N_{1}))=0$, and by (\ref{E1-term}), $E_{1}=0$. 
\end{proof}

\vskip.3in

\end{document}